\documentclass{lms}\journal{}
\endpage{15}

\iftrue
\usepackage{hyperref}
\hypersetup{final, colorlinks=true, citecolor=blue}
\else
\usepackage{url}
\usepackage{nohyperref}
\fi

\usepackage{amsmath,amssymb}
\usepackage[font=small,textfont=it,labelfont=sc,
    labelsep=period,margin=10pt]{caption}

\usepackage{calrsfs}

\newtheorem{thm}[equation]{Theorem}
\newtheorem{prop}[equation]{Proposition}
\newtheorem{lem}[equation]{Lemma}

\newnumbered{prob}[equation]{Problem}
\newnumbered{conj}[equation]{Conjecture}
\newnumbered{defn}[equation]{Definition}

\newnumbered{rmk}[equation]{Remark}

\DeclareMathOperator{\SL}{\textrm{SL}}

\DeclareMathOperator{\disc}{disc}

\DeclareMathOperator{\GL}{GL}

\DeclareMathOperator{\Hom}{Hom}

\DeclareMathOperator{\sgn}{sgn}

\newcommand{\C}{\mathbb C}

\newcommand{\HH}{\mathbb H}

\newcommand{\Q}{\mathbb Q}
\newcommand{\R}{\mathbb R}
\newcommand{\Z}{\mathbb Z}

\newcommand     {\abs}[1]       {{\left\lvert{#1}\right\rvert}}
\newcommand     {\kro}[2]       {{\left(\frac{#1}{#2}\right)}}

\newcommand{\smat}[1]{\left(\begin{smallmatrix}#1\end{smallmatrix}\right)}
\def\<#1>{\left\langle{#1}\right\rangle}

\renewcommand{\P}{\mathbb{P}}

\title{Computing Jacobi forms}
\author[N. Ryan, N. Sirolli, N.-P. Skoruppa, G. Tornar\'ia]
{Nathan C. Ryan,
Nicol\'as Sirolli,
Nils-Peter Skoruppa, and Gonzalo Tornar\'ia}

\extraline{The first author was supported by CNIC grant number IIA-1341090
from the Office of International Integrative Activities at the NSF.
\par
The second author was fully supported by the Agencia Nacional de 
Investigaci\'on e Innovaci\'on, under grant code PD\_NAC\_2013\_1\_11010}

\classno{11F50, 11-04}

\begin{document}

\maketitle

\begin{abstract}
  We describe an implementation for computing holomorphic and
  skew-holomorphic Jacobi forms of integral weight and scalar index on
  the full modular group. This implementation is based on formulas
  derived by one of the authors which express Jacobi forms in terms of
  modular symbols of elliptic modular forms. Since this method allows
  to generate a Jacobi eigenform directly from a given modular
  eigensymbol without reference to the whole ambient space of Jacobi
  forms it makes it possible to compute Jacobi Hecke eigenforms of
  large index. We illustrate our method with several examples.
\end{abstract}
\section{Introduction}

Jacobi forms play a central role in the theory of automorphic
forms (e.g., via the Fourier Jacobi expansion of orthogonal modular
forms), in quantum field theory (where they appear as characters of
infinite dimensional Lie algebras) and in algebraic geometry (where
they provide an indispensable tool for the construction of functions
with prescribed behavior of their divisors). A philosophical reason for
this might be that any given space of scalar-valued or vector-valued
elliptic modular forms of integral or half-integral weight can be
naturally embedded into a space of Jacobi forms of integral
weight and lattice index on the full modular group~\cite{S10}.

In addition to their centrality and importance, they have the striking
property that there are various methods to compute their Fourier
expansions or even to describe them by explicit formulas.  The main
techniques to compute Jacobi forms are (1) theta blocks
\cite{GSZ}, (2) pullback of Jacobi forms of lattice index of singular
or critical weight, where the latter are essentially invariants of
Weil representations of~$\mathrm{SL}(2,\Z)$ \cite{Boylan}, (3)
the Taylor expansion of a Jacobi form around $z=0$ (see e.g.~\cite{S0})
and (4) modular symbols~\cite{S1,S3,S2}.  In this
paper we describe an implementation based on the latter.

Why do we focus on the modular symbol method?  Theta blocks work
nicely for small weights and produce appealing explicit formulas but
miss more and more Jacobi forms the larger the weight
becomes. Similarly, it is not yet clear in what generality the most
recent method of pulling back singular and critical weight Jacobi forms of lattice
index works. The Taylor expansion method always works and is easy to
implement but becomes computationally harder as the index and,
accordingly, the dimension of the spaces of Jacobi forms increase. In
contrast to this, the modular symbol method allows one to compute directly a
desired Jacobi eigenform without having to generate first a whole
space of Jacobi forms and then cut it down in a second step to the
eigenspace one is looking for. More precisely, we start with a modular
symbol representing an elliptic eigen-newform we are interested
in. The articles referenced above that are related to the modular symbol method 
propose closed formulas for the Jacobi newform associated to this elliptic 
modular form. This allows, for example, for an automatic generation of Tunnell 
like formulas for (the squares of) the central values of the L-function attached 
to the twist of a given elliptic curve over the rationals.

Whereas a ready-to-compute description of Jacobi forms in terms of
modular symbols is given in~\cite{S2} for weight $k=2$ (based on the
results in~\cite{S3}), the formulas in~\cite{S1}, which cover weights
$k\ge 3$, need some reformulation along the lines of~\cite{S2}. Since
the necessary steps for this are not completely obvious we found it
worthwhile to deduce and describe these steps in detail in this
note. The resulting formulas are summarized
in~Theorem~\ref{thm:formula}.

The paper is organized as follows.  We begin by discussing a couple of 
examples. After showing the kind of data the algorithm can produce, we give 
some background to Jacobi forms and then proceed to state and prove the 
mentioned formula we implemented.  We then highlight some details of 
our implementation and we conclude with an Appendix of various tables as proof 
of concept of our formulas and implementations.The implementation of the formula 
we derive is available at~\cite{jmf_webpage}.

\section{Examples}\label{sec:egs}

We start with some examples that we computed via an implementation of 
\eqref{eq:formula_implemented}. These two examples are two whose correctness can 
be checked independently as they have been computed elsewhere in the 
literature.

First, we construct the (unique up to normalization) holomorphic Jacobi 
cuspform $\phi_{2,37}$ of weight 2 and index 37. The coefficients in this case 
are indexed by pairs of integers $n,r$ so that $r^2 < 4\cdot 37 
\cdot n$.  The $n$ are the exponents of the  $q=e^{2\pi i \tau}$ ($\tau\in\HH$) 
and the $r$ are the exponents of the $\zeta=e^{2\pi i z}$ ($z\in\C$):
\begin{align*}
\phi_{2,37}(\tau,z)\, =
\rlap{$q$}\phantom{q^2}\,&\bigl({}
- 2\,(\zeta+\zeta^{-1})
+ 4\,(\zeta^2+\zeta^{-2})
+ 3\,(\zeta^{4}+\zeta^{-4})
\\ & \phantom{\bigl(}
+ 3\,(\zeta^{5}+\zeta^{-5})
- 4\,(\zeta^{6}+\zeta^{-6})
- 4\,(\zeta^{7}+\zeta^{-7})
\\ & \phantom{\bigl(}
-    (\zeta^{8}+\zeta^{8})
+ 6\,(\zeta^{9}+\zeta^{-9})
+ 3\,(\zeta^{11}+\zeta^{-11})
+ 3\,(\zeta^{12}+\zeta^{-12})
\bigr)\\[3pt]
{}+q^2\,&\bigl(\,
  4\,(\zeta+\zeta^{-1})
+    (\zeta^{2}+\zeta^{-2})
-    (\zeta^{3}+\zeta^{-3})
- 6\,(\zeta^{4}+\zeta^{-4})
\\ & \phantom{\bigl(}
-    (\zeta^{5}+\zeta^{-5})
+ 2\,(\zeta^{6}+\zeta^{-6})
- 4\,(\zeta^{7}+\zeta^{-7})
+ 6\,(\zeta^{8}+\zeta^{-8})
\\ & \phantom{\bigl(}
+    (\zeta^{11}+\zeta^{-11})
- 2\,(\zeta^{12}+\zeta^{-12})
+    (\zeta^{13}+\zeta^{-13})
\\ & \phantom{\bigl(}
- 3\,(\zeta^{14}+\zeta^{-14})
+    (\zeta^{15}+\zeta^{-15})
+ 2\,(\zeta^{16}+\zeta^{-16})
-    (\zeta^{17}+\zeta^{-17})
\bigr)\\
{}+\rlap{$\cdots$}\phantom{q^2}\,&
\end{align*}
More coefficients can be found in Table~\ref{tbl:comparison}.  These 
coefficients can be independently checked by comparing them to \cite[Table 4]{EZ}.

Second, we construct the holomorphic cuspform $\phi_{10,1}$ of weight 10 and index 1.  In 
this case the Fourier expansion is indexed by pairs of integers $n,r$ so that 
$r^2 < 4\cdot 1 \cdot n$, with the notation otherwise as above:
\begin{align*}
\phi_{10,1}(\tau,z)\, =
\rlap{$q$}\phantom{q^2}\,&\bigl({}
- 2+(\zeta+\zeta^{-1})
\bigr)\\
{}+q^2\,&\bigl(\,{}
36
-16\,(\zeta+\zeta^{-1})
- 2\,(\zeta^{2}+\zeta^{-2})
\bigr)\\
{}+q^3\,&\bigl(\,{}
- 272
+ 99\,(\zeta+\zeta^{-1})
+ 36\,(\zeta^{2}+\zeta^{-2})
+   (\zeta^{3}+\zeta^{-3})
\bigr)\\
{}+\rlap{$\cdots$}\phantom{q^2}\,&
\end{align*}
More coefficients can be found in Table~\ref{tbl:weight}.  These coefficients  
can be independently checked by comparing them to the first coefficient in the 
Fourier--Jacobi expansion of the Siegel modular cuspform of degree 2, weight 10 
and level 1 (see, e.g., 
\cite[\href{http://beta.lmfdb.org/ModularForm/GSp/Q/Sp4Z.10_Maass/}{Maass form 
of weight 10}]{lmfdb}).

\section{Jacobi forms}\label{sect:jmf}

The basic reference for holomorphic Jacobi forms is the book of Eichler and 
Zagier \cite{EZ}, whereas for skew-holomorphic Jacobi forms we refer the reader to 
\cite{S4}.
Denote by $J_{k,m}^-$ and $J_{k,m}^+$
the spaces of holomorphic and skew-holomorphic Jacobi forms of weight
$k$ and index $m$, respectively.
Thus $J_{k,m}^\varepsilon$, for $\varepsilon=\pm 1$
and integers $k\geq 2$ and $m\geq 1$,
is the space of smooth and periodic functions $\phi(\tau, z)$ 
with $\tau\in\HH$, $z\in\C$, having a Fourier expansion of the form
\begin{equation}\label{eq:fe}
\mbox{}\qquad\qquad
\phi(\tau,z) = \sum_{\substack{\Delta,r\in\Z,\;\varepsilon\Delta \geq 0 
\\ \Delta\equiv r^2\bmod{4m}}} c_\phi(\Delta,r)\,e^{2\pi i \left(
\tfrac{r^2-\Delta}{4m} u + \tfrac{r^2+\abs{\Delta}}{4m}iv+rz\right)}
\qquad(\tau=u+iv),
\end{equation}
where the coefficients $c_\phi(\Delta,r)$ depend on $r$ only modulo $2m$ (see \cite[Thm.~2.2]{EZ}),
and such that
\[
\phi\left(\frac{-1}{\tau},\frac{z}{\tau}\right)e^{-2\pi i m \tfrac{z^2}{\tau}}=\phi(\tau,z)\cdot\begin{cases}
\tau^k & \text{ if } \varepsilon = -1,\\
\bar{\tau}^{k-1}\abs{\tau} & \text{ if }\varepsilon = + 1.\end{cases}
\]
This transformation formula, used twice, implies that
$c_\phi(\Delta,-r)=(-1)^{k-1}\,\varepsilon\,c_\phi(\Delta,r)$.
Also note that there is a skew-linear involution $\jmath:J^\varepsilon_{k,m} \to
J^\varepsilon_{k,m}$, given by $\phi(\tau,z)\mapsto 
\overline{\phi(-\overline{\tau},-\overline{z})}$,
which satisfies 
$c_{\jmath \phi}(\Delta,r) = \overline{c_\phi(\Delta,r)}$.

As mentioned above, the Fourier coefficients are indexed by pairs of
integers $(\Delta,r\bmod 2m)$ with $r^2 \equiv \Delta \pmod{4m}$ and
$\Delta\leq 0$ if $\phi$ is holomorphic and $\Delta\geq 0$ for
skew-holomorphic $\phi$.
We remark that, for holomorphic Jacobi forms like $\phi_{2,37}$ and
$\phi_{10,1}$ in Section~\ref{sec:egs}, we have
$\tfrac{r^2-\Delta}{4m} = \tfrac{r^2 + \abs{\Delta}}{4m}$ and so
an alternative is to use Fourier coefficients
$a_\phi(n,r)=c_\phi(r^2-4mn,r)$ indexed by
pairs of integers $n$, $r$ with $r^2-4mn\leq 0$ and obtain a Fourier
expansion of the form
\[
    \qquad\qquad\qquad\qquad
    \phi(\tau,z)=\sum_{n=0}^\infty q^n
    \left( \sum_{r^2\leq 4mn} a_\phi(n,r) \,\zeta^r \right)
    \qquad\qquad
    \left({q=e^{2\pi i\tau}\!,\, \zeta=e^{2\pi i z}}\right)
\]
as illustrated by those two examples.
However, in what follows it will be more convenient
to index the coefficients as in \eqref{eq:fe}, which works for both
holomorphic as well as skew-holomorphic Jacobi forms.

A form in $J^\varepsilon_{k,m}$ is cuspidal if $c(0,r)=0$ for every $r$.
For $k=2$ and $\varepsilon=+1$ there are certain
\emph{trivial cusp forms}. These are those Jacobi cusp forms
for which $c(\Delta,r)$ is non-zero at most if $\Delta$ is a perfect
square
(see the definition of $T_{r_0}$ in \cite[p. 514]{S3}). Their Fourier 
coefficients are trivial to compute.

We denote by $S^\varepsilon_{k,m}$ the subspace of $J^\varepsilon_{k,m}$
consisting of the cuspidal forms
\emph{which are orthogonal to the trivial cusp forms}.
Let $S^\varepsilon_{2k-2}(m)$ denote the space of classical holomorphic modular 
cuspforms of weight $2k-2$ for $\Gamma_0(m)$ whose L-functions have functional 
equation with sign $\varepsilon$. The following result was proved in 
\cite[Theorem 5]{SZ} when $\varepsilon = -1$. For the case $\varepsilon = 1$ it was announced
in \cite[Main Theorem]{S4}; its proof will be given in~\cite{BSZ}.

\begin{thm}\label{thm:lift}
Assume $k\ge 2$. For any fixed fundamental discriminant $\Delta_0$ and any fixed
integer $r_0$ such that $\Delta_0\equiv r_0^2\pmod{4m}$
and $\sgn\Delta_0=\varepsilon$, there is a Hecke equivariant map
\[
    \mathcal{S}_{\Delta_0,r_0} :
    S^\varepsilon_{k,m} \to S^\varepsilon_{2k-2}(m)
\]
given by
\[
\mbox{}\qquad\qquad
 \mathcal{S}_{\Delta_0,r_0}(\phi)
 = \sum_{n\geq 1} \left(
     \sum_{d\mid n} \kro{\Delta_0}{d}
     c_\phi\left(\frac{n^2}{d^2}\Delta_0,\frac{n}{d}r_0\right)
 \right) q^n \qquad (q = e^{2\pi i \tau}).
\]
Some linear combination of these maps is injective, and
its image comprises all newforms in $S^\varepsilon_{2k-2}(m)$.
Furthermore, $\mathcal{S}_{\Delta_0,r_0}$ sends
newforms to newforms.
\end{thm}

We remark that the sum of the images of the maps
$\mathcal{S}_{\Delta_0,r_0}$ can be explicitly described. It consists
of the {\em certain space} introduced in~\cite{SZ}. Moreover, the
theorem remains valid also for Eisenstein series $\phi$ (with a
suitable definition of the constant term of
$\mathcal{S}_{\Delta_0,r_0}(\phi)$).

\section{Formulas}\label{sect:formulas}

In \cite{S2}, a formula is given that takes as input a cuspidal modular symbol 
$\sigma$ of weight $2$, a fundamental discriminant $\Delta_0$ and a square root $r_0$ of 
$\Delta_0$ mod $4m$, and produces the $(\Delta,r)$-th coefficient of a Jacobi 
form of weight $2$ and index $m$ associated to $\sigma$ for every 
discriminant $\Delta$ such that $\Delta$ and $\Delta \Delta_0$ are not squares. 
The formula is given as a sum of terms involving the so-called intersection 
numbers of two geodesics, one connecting the roots of indefinite quadratic 
forms of discriminant $\Delta\Delta_0$ and one induced by $\sigma$. 

In this section we state and prove an extension of the formula in \cite{S2} 
to weight $k\geq3$. We also make it more suitable for computation. We start by 
fixing some notation and defining the intersection number.

\subsection{A pairing for polynomials}  Given a non-negative integer $w$,  let 
$\GL(2,\C)$ act on the space $\C[X,Y]_w$ of homogeneous polynomials of degree $w$ 
by
\[
(A\cdot P)(X,Y) = P(A^{-1}\left(\begin{smallmatrix}
  X\\Y\end{smallmatrix}\right)).
\]
We let $P(\alpha)=P(\alpha,1)$ for $\alpha\in\C$ and
$P(\infty)=P(1,0)$. Note that $(A\cdot P)(A\alpha)=P(\alpha)$.

Given $P_1 = \sum_{l=0}^{w} a_l X^l Y^{w-l}$ and $P_2 = \sum_{l=0}^{w} b_l X^l 
Y^{w-l}$ polynomials in $\C[X,Y]_{w}$, let
\[
 [P_1 \mid P_2] =
 \sum_{l=0}^{w} (-1)^ l \,\binom{w}{l}^{\!\!-1} \!a_l \,b_{w-l}
 \,.
\]

\begin{prop}\label{prop:poly_pairing}
 The bilinear pairing $[\cdot \mid \cdot]$ satisfies the following properties.
 \begin{enumerate}
  \item $[(xY-X)^{w}\mid P] = P(x)$ for every $P\in\C[X,Y]_{w}$ and $x\in\C$.
  \item $[P_1 \mid P_2] = (-1)^w\,[P_2 \mid P_1]$.
  \item $[A\cdot P_1 \mid A\cdot P_2] = \det(A)^{-w}\,[P_1 \mid P_2]$ 
for every $P_1, P_2\in\C[X,Y]_{w}$ and $A\in \GL(2,\C)$.
 \end{enumerate}
\end{prop}
\begin{proof}
 The first and second assertions are clear. To prove the third assertion, it 
 suffices to consider matrices $A$ of the form $\smat{a&0\\0&a'}$,
 $\smat{0&1\\1&0}$ and $\smat{1&b\\0&1}$. The first two cases are
 easy, so let $A=\smat{1&b\\0&1}$, with $b\in\C$,
 and assume $P_1=X^s Y^{w-s}$ and $P_2=X^t Y^{w-t}$.
 Then
 \begin{align*}
  [A\cdot P_1 \mid A\cdot P_2]
  & = [(X-bY)^s\,Y^{w-s} \mid (X-bY)^t\,Y^{w-t}]
  \\ & = \sum_{l=w-t}^{s} (-1)^l \binom{w}{l}^{\!\!-1}
  \binom{s}{l}\,(-b)^{s-l}  \,    
  \binom{t}{w-l}\,(-b)^{t-w+l}    
  \\ & = (-1)^{s}\,\frac{s!\,t!}{w!}
  \,\frac{b^{s+t-w}}{(s+t-w)!}
  \sum_{l=w-t}^{s} (-1)^{t-w+l} \binom{s+t-w}{t-w+l}
  \\ & = (-1)^{s}\,\frac{s!\,t!}{w!}
  \,\frac{b^{s+t-w}}{(s+t-w)!}
  \sum_{i=0}^{s+t-w} (-1)^{i} \binom{s+t-w}{i}
  \\ & = \begin{cases}
  (-1)^s\,\binom{w}{s}^{\!-1} & \text{if $s+t-w=0$,} \\
  0 & \text{otherwise.}
  \end{cases}
  \\ & = [P_1\mid P_2] \,.
 \end{align*}
 \end{proof}

\subsection{Modular symbols and the intersection number}

Let $k\geq2$ be an integer. Following \cite{St}, we denote by $\mathbb{M}_2$ the 
free abelian group generated by symbols $\{\alpha,\beta\}$ with 
$\alpha,\beta\in\P^1(\Q)$, modulo the relations
\[
 \{\alpha, \beta\} + \{\beta, \gamma\} + \{\gamma, \alpha\} = 0,
\]
and modulo any torsion. We let $\mathbb{M}_{2k-2} = \mathbb{M}_2\otimes 
\Z[X,Y]_{2k-4}$. Then $\mathbb{M}_{2k-2}$ is a $\GL(2,\Z)$-module, via
\[
 A\cdot(\{\alpha,\beta\}\otimes P) = \{A\alpha,A\beta\}\otimes A\cdot P.
\]
Let $m$ be a non-negative integer. We let $\mathbb{M}_{2k-2}(m)$ denote the 
space of modular symbols of weight $2k-2$ and level $m$, i.e. the quotient of 
$\Gamma_0(m)$-coinvariants of $\mathbb{M}_{2k-2}$. It has finite rank. 
Furthermore, $\mathbb{M}_{2k-2}(m)$ comes equipped with the action of Hecke 
operators. We denote the subspace of \emph{new modular symbols} by 
$\mathbb{M}^{\text{new}}_{2k-2}(m)$.

Given $P\in\Z[X,Y]_{2k-4}$ 
and $A\in \GL(2,\Z)$, let $[P,A]$ denote the \emph{Manin symbol}
\[
 [P,A] = A \cdot (\{0,\infty\}\otimes P) \quad \in \mathbb{M}_{2k-2}(m).
\]
By \cite[Proposition 8.3]{St} these symbols span $\mathbb{M}_{2k-2}(m)$. This 
implies that every modular cuspform $f\in S_{2k-2}(m)$ induces a Hecke 
equivariant map
\begin{align*}
 I_f :\mathbb{M}_{2k-2}(m) & \longrightarrow \C \\
 [P,A] & \mapsto \int_{0}^{i\infty} (f\vert_{2k-2}[A])(t)\,P(t,1)\,dt.
\end{align*}

Let $\Z[\P^1(\Q)]$ denote the free abelian group generated by symbols 
$(\alpha)$ with $\alpha \in \P^1(\Q)$, and let
$\mathbb{B}_{2k-2} = \Z[\P^1(\Q)]\otimes \Z[X,Y]_{2k-4}$. Then 
$\mathbb{B}_{2k-2}$ is a $\GL(2,\Z)$-module via
\[
A\cdot((\alpha)\otimes P) = (A\alpha)\otimes A\cdot P.
\]
We define $\mathbb{B}_{2k-2}(m)$ to be the quotient of 
$\Gamma_0(m)$-coinvariants of $\mathbb{B}_{2k-2}$. We have a map of 
$\GL(2,\Z)$-modules $\partial : \mathbb{M}_{2k-2}(m)\to \mathbb{B}_{2k-2}(m)$ 
induced by
\begin{equation*}
 \{\alpha,\beta\}\otimes P \mapsto ((\beta)-(\alpha)) \otimes P.
\end{equation*}
We let $\mathbb{S}_{2k-2}(m) = \ker \partial$, and we let 
$\mathbb{S}^{\text{new}}_{2k-2}(m) = \mathbb{S}_{2k-2}(m) \cap  
\mathbb{M}^{\text{new}}_{2k-2}(m)$.

For $\varepsilon = \pm 1$ we denote by $\mathbb{M}^\varepsilon_{2k-2}(m)$ 
the subspace where $g=\left(\begin{smallmatrix} -1 & 0 \\ 0 & 1 
\end{smallmatrix}\right)$ acts as multiplication by $(-1)^{k-1}\varepsilon$.
Given $\sigma \in \mathbb{M}_{2k-2}(m)$ we denote
\[
 \sigma^\varepsilon = \sigma + (-1)^{k-1} \, \varepsilon \, (g \cdot \sigma) 
\quad \in \mathbb{M}^\varepsilon_{2k-2}(m).
\]
Finally, we define 
$\mathbb{S}^\varepsilon_{2k-2}(m) = \mathbb{S}_{2k-2}(m) \cap 
\mathbb{M}^\varepsilon_{2k-2}(m)$.

Let $Q\in\Z[X,Y]_2$ be a binary quadratic form with integral coefficients.
We define an 
\emph{intersection number} map 
$C_Q : \mathbb{M}_{2k-2} \to \Q$ by
\begin{equation}\label{eq:int_number_def}
 C_Q \cdot \{\alpha, \beta\} \otimes P = \frac12 \bigl(\sgn Q(\alpha) - 
\sgn Q(\beta)\bigr) \, [P \mid Q^{k-2}].
\end{equation}
By Proposition \ref{prop:poly_pairing}, we have that $C_{A\cdot Q} \cdot (A 
\cdot \sigma) = C_Q \cdot \sigma$ for every $A \in \GL(2,\Z)$ and every $\sigma 
\in \mathbb{M}_{2k-2}$.

\begin{rmk}
 Assume that $Q$ has positive discriminant, and write $Q = (uX + vY)(wX + tY)$ 
with $u,v,w,t \in \R$. We associate to~$Q$ the {\em Heegner cycle}
\[
 C_Q = \sgn(ut-vw) \{-v/u,-t/w\} \otimes Q^{k-2}.
\]
Then $C_Q \cdot \sigma$ can be interpreted as an intersection number on 
$\C[\mathfrak{L}]\otimes \Z[X,Y]_{2k-4}$, where $\mathfrak{L}$ is the set of 
oriented hyperbolic lines in the Poincar\'e uper half plane, and where a symbol 
$\{\alpha, \beta\}$ is identified with the hyperbolic line with $\alpha$ as 
``starting point'' and $\beta$ as ``end point''.  
Furthermore, when $k=2$ the intersection number $C_Q \cdot C_R$ agrees 
with the one introduced in \cite{S2}.
\end{rmk}

\subsection{Formulas for Fourier coefficients of Jacobi forms}

For the rest of this section we assume that $(\Delta_0,r_0)$ is a fixed 
\emph{$m$-admissible pair}, i.e. $\Delta_0$ is a fundamental discriminant, and 
$r_0$ is an integer such that $\Delta_0 \equiv r_0^2\pmod{4m}$.
Furthermore, we assume that $\sgn(\Delta_0) = \varepsilon$.

The key idea of~\cite{S3} and~\cite{S1} for obtaining explicit formulas for 
Jacobi forms in terms of modular symbols is to consider the Hecke equivariant 
map
\begin{align*}
\Sigma_{\Delta_0,r_0}: S^{\varepsilon}_{k,m} 
\xrightarrow{\mathcal{S}_{\Delta_0,r_0}}
S^{\varepsilon}_{2k-2}(m)& \longrightarrow
\Hom(\mathbb{M}_{2k-2}(m)\otimes \C,\C)
\\
f&\mapsto I_f\nonumber
\end{align*}
and to dualize it. Identifying $S_{k,m}$ with its dual space via the map $\phi 
\mapsto \left\langle \cdot,\jmath \phi \right\rangle$, we get a Hecke 
equivariant map $\Sigma^*_{\Delta_0,r_0} : \mathbb{M}_{2k-2}(m)\otimes \C \to 
S^{\varepsilon}_{k,m}$ that satisfies
\begin{equation}\label{eq:sigma*}
 \left\langle \phi, \jmath\Sigma^*_{\Delta_0,r_0}(\sigma) 
\right\rangle = \Sigma_{\Delta_0,r_0}(\phi)(\sigma^\varepsilon)
\end{equation}
for every $\sigma \in \mathbb{M}_{2k-2}(m)$ and every 
$\phi \in S^{\varepsilon}_{k,m} $. By Theorem \ref{thm:lift}, since the map 
which associates to a modular form its periods is injective, 
some linear combination of the maps $\Sigma^*_{\Delta_0,r_0}$ is surjective.
Furthermore,
every newform in $S^\varepsilon_{k,m}$ can be obtained from some
$\sigma\in\mathbb{S}^{\varepsilon,\text{new}}_{2k-2}$.

From here on assume that $k\geq 3$, and let
\[
 b_{k,m} = 
\frac{2}{\sqrt{\varepsilon}}\left(\frac{2\varepsilon}{mi}\right)^{k-2},
\]
where $i=\sqrt{-1}$.

We denote $\mathcal{Q}_m = \{[ma,b,c]\;:\; a,b,c \in \Z\}$,
where $[a,b,c]$ represents the binary quadratic form
$aX^2+bXY+cY^2\in\Z[X,Y]_2$.
Denote by $\chi_{m,\Delta_0}:\mathcal{Q}_m \to \{0,1,-1\}$ the genus character 
introduced in \cite[Proposition 1]{GKZ}. Given integers $\Delta,r$, we let
\[
 \mathcal{Q}_m(\Delta,r) = \{[ma,b,c] \in \mathcal{Q}_m\;:\; b^2-4mac = 
\Delta,\;b\equiv r \mod 2m\}.
\]
Note that if $\Delta \neq \square$ and $[ma,b,c] \in \mathcal{Q}_m(\Delta,r)$ 
then $ac \neq 0$.

With this notation in mind, given $A \in \SL(2,\Z)$ we let 
$\mathcal{L}^A_{\Delta_0,r_0} : \mathbb{H}\times\C \to \C[X,Y]_{2k-4}$ be the 
kernel map defined in \cite{S1}. For fixed $x\in\C$ the function 
$\mathcal{L}^A_{\Delta_0,r_0}(\cdot)(x) $ belongs to $S^\varepsilon_{k,m}$, and 
its $(\Delta,r)$-th Fourier coefficient
is given by $b_{k,m} \, \mathcal{C}_{\Delta_0, 
r_0}^A(\Delta,r)(x)$, where $\mathcal{C}_{\Delta_0, 
r_0}^A(\Delta,r)\in\R[X,Y]_{2k-4}$ is given by
\begin{align}
 \mathcal{C}_{\Delta_0, r_0}^A(\Delta,r)(x)\; & =  
\sum_{\substack{Q\in\mathcal{Q}_m(\Delta\Delta_0,r r_0) \\ A^{-1}\cdot Q = 
[a,b,c], \,ac<0 }} \chi_{m,\Delta_0}(Q)\,\sgn(a)\,(A^{-1}\cdot Q)(x)^{k-2} 
\nonumber\\[4pt]
  & + \sum_{\substack{Q\in\mathcal{Q}_m(\Delta\Delta_0,r r_0) \\ A^{-1}\cdot Q 
= [0,b,c], \,0\leq c < N }} F_Q(x) 
    - \sum_{\substack{Q\in\mathcal{Q}_m(\Delta\Delta_0,r r_0) \\ A^{-1}\cdot Q 
= [a,b,0], \,0\leq a < N }} G_Q(x)
\label{eq:L_coeffs}\\[4pt]
& + \mathcal{Z}^A_{\Delta_0,r_0}(\Delta,r)(x) \nonumber\,.
\end{align}
Here $N$ is certain non-negative integer; $F_Q(X),G_Q(X)$ are certain 
polynomials over $\Q$ of degree at most $k-1$ (they have explicit descriptions 
which we do not need). Furthermore, $\mathcal{Z}^A_{\Delta_0,r_0}(\Delta,r)(x)$ 
is a correction term which we describe below.

The following result is proved in \cite[Proposition 4]{S1}.

\begin{prop}
The map $\mathcal{L}^A_{\Delta_0,r_0}$ is a kernel map, in the sense that it 
satisfies 
\begin{align}
 \left\langle \phi, \mathcal{L}^A_{\Delta_0,r_0} 
(\cdot)(\overline{x})\right\rangle
& = \Sigma_{\Delta_0,r_0}(\phi)([P_x,A]^\varepsilon)
\label{eq:L_and_Sigma}
\end{align}
for every $\phi\in S^\varepsilon_{k,m}$ and $x\in\C$,
where $P_x = (xY-X)^{2k-4}\in\C[X,Y]_{2k-4}$.
\end{prop}

\begin{rmk}
 In the statement of \cite[Proposition 4]{S1} there is a tiny mistake: in the 
left hand side of \eqref{eq:L_and_Sigma} the kernel map 
$\mathcal{L}^A_{\Delta_0,r_0}$ appears evaluated in $-\overline{x}$, but it 
should be $\overline{x}$.
 
\end{rmk}

The following lemma
relates $\Sigma^*_{\Delta_0,r_0}$ to the kernel map introduced above.

\begin{lem}\label{lem:sigma*_vs_kernel}
For every $P\in\Z[X,Y]_{2k-4}$ and $A\in \SL(2,\Z)$, we have that
 \begin{equation}\label{eq:L_kernel}
  \Sigma^*_{\Delta_0,r_0} ([P,A])(\tau,z) 
= \overline{b_{k,m}}/\!\lower4pt\hbox{$b_{k,m}$} \,
    [P \mid \mathcal{L}^A_{\Delta_0,r_0} (\tau,z)].
 \end{equation}
\end{lem}

\begin{proof} 
By \eqref{eq:sigma*} and \eqref{eq:L_and_Sigma} we have
\[
  \<\phi,\mathcal{L}^A_{\Delta_0,r_0}(\cdot)(\overline{x})>
= \Sigma_{\Delta_0,r_0}(\phi)([P_x,A]^\varepsilon)
= \<\phi,\jmath\Sigma^*_{\Delta_0,r_0}([P_x,A])>
\]
for every $\phi\in S^\varepsilon_{k,m}$.
Hence,
$\Sigma^*_{\Delta_0,r_0}([P_x,A])
=\jmath(\mathcal{L}^A_{\Delta_0,r_0}(\cdot)(\overline{x}))$.
By the definition given in \cite[Proposition 4]{S1}, there exist $\phi_l \in 
S^\varepsilon_{k,m}$ with real Fourier coefficients such that 
$\mathcal{L}^A_{\Delta_0,r_0}(\cdot)(x) = b_{k,m} \sum_l \phi_l\, x^l$ for every $x\in 
\C$. This implies that
\begin{equation}\label{eq:jL}
    \jmath\bigl(\mathcal{L}^A_{\Delta_0,r_0}(\cdot)(\overline{x})\bigr)
    = \overline{b_{k,m}} / \!\lower4pt\hbox{$b_{k,m}$}
    \, \mathcal{L}^A_{\Delta_0,r_0}(\cdot)(x)
\end{equation}
for every $x\in \C$. Since $P(x)=[P_x\mid P]$ for any polynomial
$P\in\C[X,Y]_{2k-4}$, we conclude
\[
  \Sigma^*_{\Delta_0,r_0} ([P_x,A])(\tau,z) 
  = \overline{b_{k,m}}/\lower2pt\hbox{$b_{k,m}$} \,
      \mathcal{L}^A_{\Delta_0,r_0}(\tau,z)(x)
  = \overline{b_{k,m}}/\lower2pt\hbox{$b_{k,m}$} \,
      [P_x \vert \mathcal{L}^A_{\Delta_0,r_0}(\tau,z)].
\]
Now \eqref{eq:L_kernel} follows by linearity, since the polynomials
$P_x$ generate $\C[X,Y]_{2k-4}$.
\end{proof}

For a pair of integers $\Delta,r$
we let $\xi_{\Delta,r}:\P^1(\Q)\to \R$ be the map given by
 \[
     \xi_{\Delta,r}(\alpha) =
     \gamma \bigl(
         \zeta_{m,\Delta\Delta_0, rr_0, \alpha,  \Delta_0}(k-1)
         + (-1)^{k-1}\,\varepsilon\,
         \zeta_{m,\Delta\Delta_0, rr_0, -\alpha,  \Delta_0}(k-1)
        \bigr)
     \,,
 \]
where $\zeta_{m,\Delta,r ,\alpha,\Delta_0}(s)$ denotes the Dirichlet series
defined in \cite[p. 67]{S1},
and where $\gamma \in \R$ is as in \cite[Proposition 4]{S1} (with $k$ replaced by 
$k-1$).
By properties of these Dirichlet series, the function $\xi_{\Delta,r}$
satisfies that $\xi_{\Delta,r}(A\alpha) = \xi_{\Delta,r}(\alpha)$
for every $A \in \Gamma_0(m)$.

For $A \in \SL(2,\Z)$, the correction term
$\mathcal{Z}^A_{\Delta_0,r_0}(\Delta,r) \in \R[X,Y]$ appearing in 
\eqref{eq:L_coeffs}
is given by
\begin{align}
   \mathcal{Z}^A_{\Delta_0,r_0}(\Delta,r) & = 
    \xi_{\Delta,r}(A0)\,X^{2k-4} - \xi_{\Delta,r}(A\infty)\,Y^{2k-4}
    \,.
    \nonumber
\end{align}

\begin{lem}\label{lem:Xi}
 For each $\Delta,r$ there is a map
 $\Xi_{\Delta,r} : 
\mathbb{B}_{2k-2}(m) \to \R$ satisfying
 \[
  \Xi_{\Delta,r}(\partial [P,A]) 
  = [P \mid \mathcal{Z}^A_{\Delta_0,r_0}(\Delta,r)]
 \]
 for every $P\in \Z[X,Y]_{2k-4} $ and $A\in \GL(2,\Z)$.
 
\end{lem}

\begin{proof}
The correspondence 
$(\alpha) \otimes P \mapsto -\xi_{\Delta,r}(\alpha)\,P(\alpha)$
induces a map $\Xi_{\Delta,r} : \mathbb{B}_{2k-2}(m) \to \R$,
since $\xi_{\Delta,r}(A\alpha) = \xi_{\Delta,r}(\alpha)$
for every $A \in \Gamma_0(m)$.

Given $P\in \Z[X,Y]_{2k-4} $ and $A\in \GL(2,\Z)$,
we have
\begin{align*}
\Xi_{\Delta,r}(\partial [P,A])
    & = \Xi_{\Delta,r}(((A \infty) - (A 0))\otimes A\cdot P)
    = \xi_{\Delta,r}(A 0)\,P(0) - \xi_{\Delta,r}(A\infty)\,P(\infty)
    \\ & = \xi_{\Delta,r}(A0)\,[P\mid X^{2k-4}] - \xi_{\Delta,r}(A\infty)\,[P\mid Y^{2k-4}]
 \\ & = [P\mid \mathcal{Z}^A_{\Delta_0,r_0}(\Delta,r)]
 \,,
\end{align*}
where we used that
$[P \mid X^{2k-4}] = P(0)$ and $[P \mid Y^{2k-4}] = P(\infty)$.
\end{proof}

\begin{prop}\label{prop:formula}
Let $P\in\Z[X,Y]_{2k-4}$ and let $A\in \SL(2,\Z)$.
Let $\phi \in S^\varepsilon_{k,m}$ be given by $\phi = 
1/\lower2pt\hbox{$\overline{b_{k,m}}$}\,
\Sigma^*_{\Delta_0,r_0}([P,A])$. Then for every 
$\Delta$ such that $\Delta\Delta_0\neq \square$
we have
\begin{align}\label{eq:coeffs_prop}
c_{\phi}(\Delta,r) = 
\sum_{\substack{Q\in\mathcal{Q}_m(\Delta\Delta_0,r r_0) \\ A^{-1}\cdot Q =
    [a,b,c], \,ac<0 }} \chi_{m,\Delta_0}(Q)\,\sgn(a)\,[A \cdot P\mid Q^{k-2}] + 
\Xi_{\Delta,r}(\partial [P,A]).
\end{align}
\end{prop}

\begin{proof}
Since $\Delta\Delta_0$ is not a square, the second and third 
summands of $\mathcal{C}_{\Delta_0, r_0}^A(\Delta,r)(x)$ in \eqref{eq:L_coeffs} 
are empty.
Hence using Lemmas \ref{lem:sigma*_vs_kernel} and \ref{lem:Xi}, and 
\eqref{eq:L_coeffs}, we get that
\begin{align*}
   c_{\phi}(\Delta,r) & = [P \mid \mathcal{C}_{\Delta_0, r_0}^A(\Delta,r)] \\
& =  
\sum_{\substack{Q\in\mathcal{Q}_m(\Delta\Delta_0,r r_0) \\ A^{-1}\cdot Q = 
[a,b,c], \,ac<0 }} \chi_{m,\Delta_0}(Q) \sgn(a)[P\mid (A^{-1}\cdot Q)^{k-2}] 
  + \Xi_{\Delta,r}(\partial[P,A]).
\end{align*}
Since by Proposition \ref{prop:poly_pairing} we have that $[P\mid (A^{-1}\cdot 
Q)^{k-2}] = [A\cdot P\mid Q^{k-2}]$, this completes the proof.
\end{proof}

With these preliminary results and notation in hand we now prove the main
formula, which extends \cite[Theorem 3]{S2} to weights $k\ge 3$.

\begin{thm}\label{thm:formula} Given $\sigma \in \mathbb{M}_{2k-2}(m)$,
let $\phi = -1/\lower2pt\hbox{$\overline{b_{k,m}}$}\,
\Sigma^*_{\Delta_0,r_0}(\sigma)$. Then for every $\Delta$ such that 
$\Delta\Delta_0\neq \square$ we have that 
\begin{equation}\label{eq:formula-k}
c_{\phi}(\Delta,r) = 
\sum_{Q\in\mathcal{Q}_m(\Delta\Delta_0,r r_0)} 
\chi_{m,\Delta_0}(Q) \, C_Q \cdot \sigma 
+ \Xi_{\Delta,r}(\partial\sigma).
\end{equation}
\end{thm}

\begin{rmk}
 When $\sigma$ is cuspidal we have that $\Xi_{\Delta,r}(\partial\sigma) = 
\Xi_{\Delta,r}(0) = 0$, and hence the right hand side of \eqref{eq:formula-k} 
becomes simpler. In particular, we do not need to compute
the Dirichlet series appearing in the definition of
$\xi_{\Delta,r}(\alpha)$.  For completeness, though, we observe that, for level 
1, these Dirichlet series are partial zeta functions of quadratic number fields, 
and for higher level, when certain congruences in the summation have to be 
observed, they become partial ray class zeta functions.  And so, in principle, 
the required special values could be computed.
\end{rmk}

\begin{rmk}
Each summand in the right hand side of \eqref{eq:formula-k} is well defined 
for $\sigma \in \mathbb{M}_{2k-2}$, but not for $\sigma \in 
\mathbb{M}_{2k-2}(m)$. 
However, since $\Gamma_0(m)$ acts on $\mathcal{Q}_m(\Delta\Delta_0,r r_0)$, by 
the $\Gamma_0(m)$-invariance of both the intersection number and the genus 
character, the right hand side of \eqref{eq:formula-k} is well defined for 
$\sigma \in \mathbb{M}_{2k-2}(m)$.
\end{rmk}

\begin{rmk}
 The proof given is based on results that require the hypothesis $k\geq3$ (for 
example, \cite[Proposition 4]{S1}). When $k=2$, Theorem \ref{thm:formula} is 
valid if we assume furthermore that $\sigma \in \mathbb{S}_{2k-2}(m)$ and that 
$\Delta \neq \square$; this is proved in \cite[Theorem 3]{S2}.
\end{rmk}

\begin{proof}
We assume without loss of generality that $\sigma=[P,A]$.
Let $Q\in\mathcal{Q}_m(\Delta\Delta_0,r r_0)$, and denote $A^{-1}\cdot Q$ by $[a,b,c]$.
Note that $ac\neq 0$, since $\disc(A\cdot Q) = \disc Q \neq \square$.
Since $Q(A0)=c$ and $Q(A\infty)=a$, we have
\begin{align*}
     C_Q\cdot[P,A]
     &= C_Q\cdot(\{A0,A\infty\}\otimes A\cdot P)
     = \frac{\sgn(c)-\sgn(a)}2
       \,[A\cdot P\mid Q^{k-2}].
\end{align*}
We have that $\sgn(c)-\sgn(a)$ is non-zero if and only if $ac < 0$,
and in that case it equals $-2\sgn(a)$.
Hence
\[
 C_Q \cdot [P,A] = 
 \begin{cases}
     -\sgn(a) \,[A\cdot P\mid Q^{k-2}]
         & ac < 0, \\
     0, & \text{otherwise}.
 \end{cases}
\]
Summing over all $Q\in\mathcal{Q}_m(\Delta\Delta_0,r r_0)$
we conclude
\[
\sum_{Q\in\mathcal{Q}_m(\Delta\Delta_0,r r_0)}
\chi_{m,\Delta_0}(Q) \, C_Q \cdot [P,A]
= -
\sum_{\substack{Q\in\mathcal{Q}_m(\Delta\Delta_0,r r_0) \\
        A^{-1}\cdot Q = [a,b,c],\,ac<0}}
\chi_{m,\Delta_0}(Q) \,
\sgn(a) \,[A\cdot P\mid Q^{k-2}]
\]
The result now follows from Proposition~\ref{prop:formula}
\end{proof}

\section{Details on the implementation}\label{sec:implementation}

In Theorem \ref{thm:formula} we stated a formula for computing Fourier 
coefficients of Jacobi forms. In this section we describe the 
practical issues related to carrying out computations using that formula. In 
particular, we rewrite this formula without making explicit mention of 
intersection numbers, we describe why the support for the infinite sums in the 
formula is finite and we identify which quadratic forms we need to consider when 
computing with the formula. We also describe a few auxiliary things we 
implemented. The code is available at 
\cite{jmf_webpage}.

Again for this section, we assume that $(\Delta_0,r_0)$ is a fixed 
\emph{$m$-admissible pair} and that $\sgn(\Delta_0) = \varepsilon$.

\subsection*{Ready-to-compute formulas}  The formulas above are appealing but not so 
useful for computation.  In the following lemma  we give the formula that we 
implement for computing the intersection numbers appearing in 
\eqref{eq:formula-k}.

Given $\sigma \in \mathbb{M}_{2k-2}$, using that $\{\alpha,\beta\} = 
\{\infty,\beta\} - \{\infty,\alpha\}$, we can write
\begin{equation}\label{eq:sigma_decomp}
 \sigma = \sum_i n_i\,\{\infty,s_i\}\otimes P_i \,.
\end{equation}

\begin{lem}\label{lem:inters_skoruppa}
Let $\sigma \in \mathbb{M}_{2k-2}$ be as in \eqref{eq:sigma_decomp},
and let $Q$ be a binary quadratic form with integral coefficients such that 
$\disc Q \neq \square$.  Then 
\begin{equation}\label{eq:inters_skoruppa}
 C_Q \cdot \sigma =
 \sgn Q(\infty)
 \sum_{Q(\infty)\,Q(s_i) < 0}
      n_i\,[P_i \mid Q^{k-2}].
\end{equation}
\end{lem}

\begin{proof}
By definition of the intersection number map $C_Q$ we have that
\[
 C_Q \cdot \sigma
 = \sum_i n_i\,C_Q\cdot(\{\infty,s_i\}\otimes P_i)
 = \sum_i n_i\,\frac{\sgn Q(\infty)-\sgn Q(s_i)}2\,[P_i \mid Q^{k-2}]
 \,.
\]
Since $\disc Q \neq \square$, we have $Q(\infty)\,Q(s_i)\neq 0$, hence
$\frac{\sgn Q(\infty)-\sgn Q(s_i)}2$ is non-zero if and only if 
$Q(\infty)\,Q(s_i)<0$, and in that case it equals $\sgn Q(\infty)$.
\end{proof}

\begin{prop}\label{prop:formulas_implemented} Given $\sigma \in 
\mathbb{S}_{2k-2}(m)$ as in \eqref{eq:sigma_decomp}, 
let $\phi = -1/\lower2pt\hbox{$\overline{b_{k,m}}$}\,
\Sigma^*_{\Delta_0,r_0}(\sigma)$. Then for every $\Delta$ such that 
$\Delta\Delta_0\neq \square$ we have that $c_{\phi}(\Delta,r) = 
\widetilde{c}(\Delta,r) + (-1)^{k-1}\,\varepsilon \,
\widetilde{c}(\Delta,-r)$, where
\begin{equation}\label{eq:formula_implemented}
\widetilde{c}(\Delta,r) = \sum_i n_i
\sum_{\substack{Q\in\mathcal{Q}_m(\Delta\Delta_0,r 
r_0) \\ Q(\infty)>0, \, Q(s_i)<0}} \chi_{m,\Delta_0}(Q) \, [P_i 
\mid Q^{k-2}].
\end{equation}
\end{prop}

\begin{proof}
 Combining Lemma \ref{lem:inters_skoruppa} and Theorem \ref{thm:formula} we get 
that
\[
 c_\phi(\Delta,r) = \sum_{i} 
n_i \sum_{\substack{Q\in\mathcal{Q}_m(\Delta\Delta_0,r 
r_0) \\ Q(\infty)\,Q(s_i)<0}}\sgn Q(\infty) \chi_{m,\Delta_0}(Q) \, [P_i 
\mid Q^{k-2}].
\]
Splitting the inner sum above, it is enough to show that
\[
  \sum_{\substack{Q\in\mathcal{Q}_m(\Delta\Delta_0,r 
r_0) \\ Q(\infty)<0, \, Q(s_i)>0}} - \chi_{m,\Delta_0}(Q) \, [P_i 
\mid Q^{k-2}]
= (-1)^{k-1} \,\varepsilon \, \sum_{\substack{Q\in\mathcal{Q}_m(\Delta\Delta_0,-r 
r_0) \\ Q(\infty)>0, \, Q(s_i)<0}} \chi_{m,\Delta_0}(Q) \, [P_i 
\mid Q^{k-2}].
\]
This follows by considering the bijection between the supports given by $Q 
\mapsto -Q$.
\end{proof}

\subsection*{Quadratic forms in the support}

The sums in \eqref{eq:formula_implemented} are indexed by indefinite quadratic 
forms in $\mathcal{Q}_m(\Delta\Delta_0,r r_0)$  with $Q(\infty)>0$ and $Q(s)<0$. 
The following lemma shows that these sets are finite, and also gives explicit 
bounds for the coefficients of the quadratic forms that we need to consider.

\begin{lem}
Let $D_{\max} > 0$. Let $Q=[a,b,c]$ be a quadratic form with $a > 0$. Assume 
that $0 < \disc Q \leq D_{\max}$. 
Further, let $s=\frac{p}{q}\in\Q$ and $Q(s)<0$. Then
\begin{align*}
a &\leq \frac{D_{\max} q^2}{4},\\ 
\left \lfloor -2as - \sqrt{D_{\max}} \right \rfloor < b & < \left \lceil - 
2as + \sqrt{D_{\max}} \right \rceil, \\
\left \lceil \frac{b^2-D_{\max}}{4a} \right \rceil  \leq c & < \left \lceil
\frac{b^2 - D_0}{4a}\right \rceil,
\end{align*}
where $D_0 =(b+2as)^2$.
\end{lem}

\begin{proof}
By hypothesis, $q^2 Q(s)$ is a negative integer. Furthermore,  $q^2 Q(s) =  a 
(p + \frac{bq}{2a})^2 - \frac{D q^2}{4a} \geq - \frac{D q^2}{4a}$, where $D = 
\disc Q$. In particular $-1 \geq -\frac{Dq^2}{4a}$, which proves the first 
inequality. The inequalities involving $b$ follow from the fact that 
$\frac{-b-\sqrt{D}}{2a} < s < \frac{-b+\sqrt{D}}{2a}$. The lower bound on $c$ 
follows easily from the bound on $D$. The upper bound on $c$ is equivalent to 
the inequality $D_0 < D$. The latter follows from the fact that
\[
D_0 = D + 4a (a s^2 + bs + c) = D + 4a Q(s).
\]
\end{proof}

\subsection*{Every coefficient can be computed}

As mentioned before, in order to compute the coefficient $c(\Delta,r)$, the 
formulas given by \eqref{eq:formula_implemented}, require that 
$\Delta\Delta_0\neq \square$.  A reasonable question to ask, then, is whether or 
not it is possible to compute every coefficient of a given Jacobi form.  The 
following lemma answers the question in the affirmative.

\begin{lem}
Given an $m$-admissible pair $(\Delta,r)$ there exists an $m$-admissible pair  
$(\Delta_1,r_1)$ with $\Delta_1$ a negative fundamental discriminant such that 
$\Delta \Delta_1 \neq \square$.
\end{lem}
\begin{proof}Let $(\Delta_0,r_0)$ be any $m$-admissible pair with $\Delta_0$ a 
negative, odd fundamental discriminant. If $\Delta \Delta_0 \neq \square$, we 
are done. Otherwise, let $p$ be a prime such that $p \equiv 1 \mod \gcd(4m, 
\Delta_0, \Delta)$. Let $\Delta_1 = p \Delta_0$ and $r_1 = r_0$. Then $\Delta_1$ 
is square-free and $\Delta_1 \equiv 1\mod 4$, whence it is a negative 
fundamental discriminant.  Furthermore, $\Delta_1 \equiv \Delta_0 \equiv r_1^2 
\mod 4m$, and $\Delta_1 \Delta = p \Delta_0 \Delta \neq \square$. Hence 
$(\Delta_1,r_1)$ satisfies the required conditions.
\end{proof}

\subsection*{Choice of $m$-admissible pair}  

The starting point for the formula in Proposition~\ref{prop:formulas_implemented} is a choice of $m$-admissible pair $(\Delta_0,r_0)$.  We find $\Delta_0$ among the negative fundamental discriminants that are squares modulo $4m$ if the form we are computing is holomorphic and among the positive fundamental discriminants that are squares modulo $4m$ if the form is skew-holomorphic.  Once a $\Delta_0$ is chosen, we compute all the square roots of $\Delta_0$ modulo $4m$ that are less than $2m$.  In practice we choose the smallest $\Delta_0$ we can in order to keep the support of the sums in \eqref{eq:formula_implemented} as small as possible.

\subsection*{Genus character}  The last part of formula 
\eqref{eq:formula_implemented} we have not described is the genus character.  We 
remark that the implementation of the formula for the genus character 
$\chi_{m,\Delta_0}(Q)$ as described in \cite[Proposition 1]{GKZ} is 
straightforward.

\subsection*{Effectiveness of the implementation}

In this section we make some brief comments about range of Jacobi forms we have 
computed using our implementation of the formulas proved above and an idea of 
the time involved to carry out some of the computations.  Timings are summarized 
in Table~\ref{tbl:timings}.  The computations are done using Sage 6.7 
\cite{Sage} and the Cython code posted at \cite{jmf_webpage}. The code was run 
on an Intel 2.7 GHz processor running RHEL 7.0.0.

In order to illustrate the effectiveness of our method, we fixed the weight at 
2 and computed expansions of forms in various indices.  We also fixed the index 
at 1 and computed expansions of forms of various weights.  As 
Table~\ref{tbl:timings} shows, were able to compute fairly quickly in weight 2 
and fairly high index and in index 1 and fairly high weight.

\begin{table}
\begin{tabular}{lll}
&& Sample time to compute $c(\Delta,r)$ for $\abs{\Delta}<\Delta_{\text{max}}$ \\
Space & $\Delta_{\text{max}}$ &  for an element of the space\\\hline
$S_{2,37}^-$ & $10000$ & $38.5$ s\\
$S_{2,11}^+$ & $10000$ & $32.7$ s\\
$S_{2,15}^+$ & $1000$ & $9.51$ s\\
$S_{2,389}^+$ & $10000$ & $37.2$ s\\
$S_{2,5077}^+$ & $100$ & $745$ s\\
$S_{10,1}^-$ & $1000$ & $87.4$ s\\
$S_{40,1}^-$ & $100$ & $7.51$ s\\
$S_{50,1}^-$ & $100$ & $9.1$ s\\
$S_{100,1}^-$ & $100$ & $20.8$ s   
\end{tabular}

\caption{Timings for the computation of particular examples of Jacobi forms, not 
including the time to compute the modular symbol.  We do not identify which 
particular form in each space we compute, but instead aim to illustrate how the 
timing depends on the weight and the index.  We do point of that if a form in 
the Appendix is an element of a space in the above table, the timing reported is 
for that form.}\label{tbl:timings}

\end{table}

\oneappendix
\section{Tables of coefficients}

In this article we have given formulas for the Fourier expansion of Jacobi 
forms.  Consider a Jacobi form $\phi$ of weight $k$ and index 
$m$.  Then $\phi$ has a Fourier expansion of the form
\begin{equation*}
\mbox{}\qquad\qquad
\phi(\tau,z) = \sum_{\substack{\Delta,r\in\Z,\;\varepsilon\Delta \geq 0 
\\ \Delta\equiv r^2\bmod{4m}}} c_\phi(\Delta,r)\,e^{2\pi i \left(
\tfrac{r^2-\Delta}{4m} u + \tfrac{r^2+\abs{\Delta}}{4m}iv+rz\right)}
\qquad(\tau=u+iv),
\end{equation*}
where $\varepsilon=-1$ if $\phi$ is holomorphic and $\varepsilon=1$ if $\phi$ 
is skew-holomorphic. The way our formula \eqref{eq:formula_implemented} works 
is that it produces the coefficients $c_\phi(\Delta,r)$.
We point out that the coefficients $c_\phi(\Delta,r)$ 
depend only on $\Delta$ and $r\pmod{2m}$; if $k$ is even and $m=1$ or 
prime, they only depend on $\Delta$. Furthermore, they satisfy that
$c_\phi(\Delta,-r)=(-1)^{k-1}\,\varepsilon\,c_\phi(\Delta,r)$.

\begin{table}[t]
\begin{tabular}{rrrr}
$\Delta$ & $r$ & $c_{-4,12}(\Delta,r)$ &  $c_{-3,21}(\Delta,r)$\\\hline
$ -3 $  &  $ 21 $  &  $ 1 $ & NA\\
$ -4 $ & $12$ & NA & 1\\
$ -7 $  &  $ 17 $  &  $ -1 $ & $-1$ \\
$ -11 $  &  $ 27 $  &  $ 1 $ & $1$ \\
$ -12 $  &  $ 32 $  &  $ -1 $ & NA\\
$-16$ & $24$ & NA & $-2$\\
$ -27 $  &  $ 11 $  &  $ -3 $ & NA \\
$ -28 $  &  $ 34 $  &  $ 3 $ & $3$\\
$-36$ & $36$ & NA & $-2$\\
$ -40 $  &  $ 16 $  &  $ 2 $ & $2$\\
$ -44 $  &  $ 20 $  &  $ -1 $ & $-1$ \\
$ -47 $  &  $ 29 $  &  $ -1 $ & $-1$ \\
$ -48 $  &  $ 10 $  &  $ 0 $ & NA\\
\end{tabular}
\caption{These are the coefficients $c(\Delta,r)$ of the holomorphic Jacobi 
cuspform in $S^-_{2,37}$ corresponding to the modular symbol
$\{\infty,-1/23\} - \{\infty,-1/32\} + \{\infty,-1/34\}- 
\{\infty,0\}\in\mathbb{S}^-_2(37)$. The data in the third column is 
the result of using \eqref{eq:formula_implemented} with $\Delta_0=-4$ and 
$r=12$.  The data in the fourth column is the result of using 
\eqref{eq:formula_implemented} with $\Delta_0=-3$ and $r=21$.  An NA means that 
$\Delta \Delta_0=\square$ and our formula does not apply.  The values of the 
coefficients agree with those in 
\protect\cite{EZ}.}
\label{tbl:comparison}
\end{table}

\begin{table}[t]
\begin{tabular}{rrr}
$\Delta$ & $r$ & $c(\Delta,r)$ \\\hline
$ 1 $  &  $ 1 $  &  $ 1 $ \\
$ 4 $  &  $ 2 $  &  $ -3 $ \\
$ 5 $  &  $ 7 $  &  $ 5 $ \\
$ 9 $  &  $ 3 $  &  $ -2 $ \\
$ 12 $  &  $ 10 $  &  $ 5 $ \\
$ 16 $  &  $ 4 $  &  $ 4 $ \\
$ 20 $  &  $ 8 $  &  $ 5 $ \\
$ 25 $  &  $ 5 $  &  $ 0 $ \\
$ 33 $  &  $ 11 $  &  $ 0 $ \\
$ 36 $  &  $ 6 $  &  $ 6 $ \\
$ 37 $  &  $ 9 $  &  $ 5 $ \\
$ 44 $  &  $ 0 $  &  $ 0 $ \\
$ 45 $  &  $ 1 $  &  $ 0 $ \\
$ 48 $  &  $ 2 $  &  $ 10 $ \\
$ 49 $  &  $ 7 $  &  $ -3 $ \\
\end{tabular}
\caption{These are the coefficients $c(\Delta,r)$ of the skew-holomorphic Jacobi 
cuspform in $S^+_{2,11}$ corresponding to the modular symbol 
$\{\infty,-1/9\} - 2\{\infty,-1/8\} + \{\infty,0\} 
\in\mathbb{S}^+_2(11)$.}\label{tbl:skew}
\end{table}

\begin{table}[t]
\begin{tabular}{rrr}
$\Delta$ & $r$ & $c(\Delta,r)$ \\\hline
$ -3 $  &  $ 1 $  &  $ -1 $ \\
$ -4 $  &  $ 0 $  &  $ 2 $ \\
$ -7 $  &  $ 1 $  &  $ 16 $ \\
$ -8 $  &  $ 0 $  &  $ -36 $ \\
$ -11 $  &  $ 1 $  &  $ -99 $ \\
$ -12 $  &  $ 0 $  &  $ 272 $ \\
$ -15 $  &  $ 1 $  &  $ 240 $ \\
$ -16 $  &  $ 0 $  &  $ -1056 $ \\
$ -19 $  &  $ 1 $  &  $ 253 $ \\
$ -20 $  &  $ 0 $  &  $ 1800 $ \\
$ -23 $  &  $ 1 $  &  $ -2736 $ \\
$ -24 $  &  $ 0 $  &  $ 1464 $ \\
$ -27 $  &  $ 1 $  &  $ 4284 $ \\
$ -28 $  &  $ 0 $  &  $ -12544 $ \\
$ -31 $  &  $ 1 $  &  $ 6816 $ \\
$ -32 $  &  $ 0 $  &  $ 19008 $ \\
$ -35 $  &  $ 1 $  &  $ -27270 $ \\
$ -36 $  &  $ 0 $  &  $ 4554 $ \\
$ -39 $  &  $ 1 $  &  $ 6864 $ \\
$ -40 $  &  $ 0 $  &  $ -39880 $ \\
$ -43 $  &  $ 1 $  &  $ 66013 $ \\
$ -44 $  &  $ 0 $  &  $ 26928 $ \\
$ -47 $  &  $ 1 $  &  $ -44064 $ \\
$ -48 $  &  $ 0 $  &  $ -12544 $ \\
\end{tabular}
\caption{These are the coefficients $c(\Delta,r)$ of the holomorphic Jacobi 
cuspform in $S^-_{10,1}$ corresponding to the modular symbol 
$\{\infty,0\}\otimes X^{14}Y^2 \in\mathbb{S}^-_{18}(1)$.  These coefficients 
agree with the first coefficient of the Fourier--Jacobi expansion of the Siegel 
modular cuspform of degree 2, weight 10 and level 1.}\label{tbl:weight}
\end{table}

\begin{table}[t]
\begin{tabular}{rrr}
$\Delta$ & $r$ & $c(\Delta,r)$ \\\hline
$ 1 $  &  $ 1 $  &  $ 1 $ \\
$ 1 $  &  $ 11 $  &  $ 1 $ \\
$ 4 $  &  $ 2 $  &  $ -2 $ \\
$ 4 $  &  $ 8 $  &  $ 2 $ \\
$ 9 $  &  $ 3 $  &  $ -2 $ \\
$ 16 $  &  $ 4 $  &  $ 0 $ \\
$ 16 $  &  $ 14 $  &  $ 0 $ \\
$ 21 $  &  $ 9 $  &  $ 8 $ \\
$ 24 $  &  $ 12 $  &  $ 8 $ \\
$ 25 $  &  $ 5 $  &  $ 0 $ \\
$ 36 $  &  $ 6 $  &  $ 4 $ \\
$ 40 $  &  $ 10 $  &  $ 0 $ \\
$ 45 $  &  $ 15 $  &  $ 0 $ \\
$ 49 $  &  $ 7 $  &  $ -1 $ \\
$ 49 $  &  $ 13 $  &  $ 1 $ \\
\end{tabular}
\caption{These are the coefficients $c(\Delta,r)$ of the skew-holomorphic Jacobi 
cuspform in $S^+_{2,15}$ corresponding to the modular symbol $\{\infty,1/5\} + 
\{\infty, -1/2\} - \{\infty,-2/5\} - \{\infty,0\} \in\mathbb{S}^+_2(15)$.  
The values of the coefficients are consistent with those in \protect\cite{PT}.}
\label{tbl:composite}
\end{table}

\begin{table}[t]
\begin{tabular}{rrr}
$\Delta$ & $r$ & $c(\Delta,r)$ \\\hline
$ 1 $  &  $ 1 $  &  $ 0 $ \\
    $ 4 $  &  $ 2 $  &  $ 0 $ \\
$ 5 $  &  $ 303 $  &  $ 1 $ \\
    $ 9 $  &  $ 3 $  &  $ 0 $ \\
$ 13 $  &  $ 205 $  &  $ -1 $ \\
    $ 16 $  &  $ 4 $  &  $ 0 $ \\
$ 17 $  &  $ 79 $  &  $ -1 $ \\
    $ 20 $  &  $ 172 $  &  $ 1 $ \\
$ 24 $  &  $ 56 $  &  $ 1 $ \\
    $ 25 $  &  $ 5 $  &  $ 0 $ \\
$ 28 $  &  $ 240 $  &  $ -1 $ \\
    $ 36 $  &  $ 6 $  &  $ 0 $ \\
$ 41 $  &  $ 279 $  &  $ -1 $ \\
$ 44 $  &  $ 40 $  &  $ -1 $ \\
    $ 45 $  &  $ 131 $  &  $ -1 $ \\
    $ 49 $  &  $ 7 $  &  $ 0 $ \\
    $ 52 $  &  $ 368 $  &  $ -1 $ \\
    $ 64 $  &  $ 8 $  &  $ 0 $ \\
$ 65 $  &  $ 125 $  &  $ 0 $ \\
    $ 68 $  &  $ 158 $  &  $ 3 $ \\
$ 69 $  &  $ 153 $  &  $ -1 $ \\
$ 73 $  &  $ 97 $  &  $ -1 $ \\
$ 76 $  &  $ 290 $  &  $ 1 $ \\
$ 77 $  &  $ 323 $  &  $ 1 $ \\
    $ 80 $  &  $ 344 $  &  $ 0 $ \\
    $ 81 $  &  $ 9 $  &  $ 0 $ \\
$ 85 $  &  $ 181 $  &  $ -1 $ \\
$ 93 $  &  $ 69 $  &  $ 0 $ \\
    $ 96 $  &  $ 112 $  &  $ -2 $ \\
$ 97 $  &  $ 137 $  &  $ 1 $ \\
\end{tabular}
\caption{These are the coefficients $c(\Delta,r)$ of the skew-holomorphic Jacobi 
cuspform in $S^+_{2,389}$ corresponding to the unique rational newform in $S^+_2(389)$,
which in turn corresponds to the elliptic curve $E$ with Cremona label \texttt{389a1}.
Note that for a fundamental discriminant $\Delta$,
 $c(\Delta,r)$ vanishes if and only if the twist $E_\Delta$ has positive
rank, in accordance with the Birch and Swinnerton-Dyer Conjecture.}
\end{table}

\newcommand{\etalchar}[1]{$^{#1}$}

\affiliationone{Nathan C. Ryan\\
Bucknell University\\
Lewisburg, PA USA}
\affiliationtwo{Nicol\'as Sirolli\\
Universidad de la Rep\'ublica,\\
Montevideo, Uruguay\\
Current affiliation:\\
Universidad de Buenos Aires\\
Buenos Aires, Argentina}
\affiliationthree{Nils-Peter Skoruppa\\
Universit\"at Siegen\\
Siegen, Germany}
\affiliationfour{Gonzalo Tornar\'ia\\
Universidad de la Rep\'ublica,\\
Montevideo, Uruguay}

\end{document}